\documentclass[12pt]{amsart}
\usepackage{amsmath}
\usepackage[all]{xy}
\usepackage{amssymb}
\usepackage{amsthm}
\usepackage{hyperref}
\usepackage{todonotes}
\hypersetup{colorlinks=true,linkcolor=blue,citecolor=magenta}
\usepackage{mathtools}
\usepackage{amscd, enumerate}
\usepackage{verbatim}
\usepackage{eurosym}
\usepackage{float}
\usepackage{color}
\usepackage{physics}
\usepackage{cases}
\usepackage{bm}
\usepackage{dcolumn}
\usepackage{makecell}
\usepackage[mathscr]{eucal}
\usepackage{bbm}
\usepackage{comment}
\usepackage{adjustbox}
\usepackage{caption, tabu, multicol}
\usepackage[textheight=8.77in, textwidth=6.8in]{geometry}

\usepackage[]{breqn}

\newtheorem{thm}{Theorem}[subsection]

\newtheorem*{thm*}{Theorem}
\newtheorem*{conj*}{Conjecture}

\newcommand{\Z}{\mathbb{Z}}

\newcommand{\sigodd}{\sigma_{\mathrm{odd}}}

\makeatletter
\DeclarePairedDelimiterX{\pmodx}[1]{(}{)}{{\operator@font mod}\mkern6mu#1}
\renewcommand{\pmod}{%
  \allowbreak
  \if@display\mkern18mu\else\mkern8mu\fi
  \pmodx
}
\makeatother

\begin{document}
\title{A Proof of Merca's Conjectures on Sums of Odd Divisor Functions}

\author[K. Lakein, A. Larsen]{Kaya Lakein and Anne Larsen}
\address{Department of Mathematics, Stanford University, Stanford, CA 94305}
\email[K. Lakein]{epi2@stanford.edu}
\address{Department of Mathematics, Harvard University, Cambridge, MA 02138}
\email[A. Larsen]{larsen@college.harvard.edu}

\begin{abstract}
In a recent paper, Merca posed three conjectures on congruences for specific convolutions of a sum of odd divisor functions with a generating function for generalized $m$-gonal numbers. Extending Merca's work, we complete the proof of these conjectures.
\end{abstract}

\maketitle

Euler's partition function $p(n)$ is defined by the number of partitions of any nonnegative integer $n$, and its generating function is given by
\begin{align*}
\sum_{n=0}^\infty p(n)q^n = \prod_{k=1}^\infty \frac{1}{1-q^k}, \quad \abs{q} < 1.
\end{align*}
The properties of the function $p(n)$, such as its asymptotic behavior and its parity, have been an object of study for a long time. For instance, Ballantine and Merca \cite{BM} recently made a conjecture on when $$\sum_{ak+1 \text{ is a square}}p(n-k)$$ is odd,
which was proved by Hong and Zhang \cite{HZ}.

The function $p(n)$ is linked to the divisor function 
$$\sigma(n) := \sum_{d \mid n}d,$$ whose generating function is given by 
\begin{align*}
\sum_{n=1}^\infty \sigma(n)q^n = \sum_{n=1}^\infty \frac{nq^n}{1-q^n}.
\end{align*}
In particular, $p(n)$ and $\sigma(n)$ satisfy the following convolution identities, which differ only in the values of $p(0)$ and $\sigma(0)$:
\begin{align*}
\sum_{k=-\infty}^\infty (-1)^kp(n-P_5(k)) &= \delta_{0,n}, \text{ with } p(0) = 1, \\
\sum_{k=-\infty}^\infty (-1)^k\sigma(n-P_5(k)) &= 0, \text{ with } \sigma(0) \text{ replaced by } n,
\end{align*}
where $\delta_{ij}$ is the Kronecker delta, and $P_m(k)$ is the $k$th generalized $m$-gonal number
\begin{align}\label{eq: m-gonal}
    P_m(k) := \left(\frac m2 - 1\right)k^2 - \left(\frac m2 - 2 \right) k.
\end{align}

Motivated by these identities as well as the fact that the divisor functions $\sigma(n)$ and 
\begin{equation*}
    \sigodd(n) := \sum_{\substack{d \mid n \\ d \text{ odd}}} d,
\end{equation*}
where $\sigodd(n) :=0$ for $n \le 0$, have the same parity, Merca recently studied the relationship between $\sigodd(n)$ and the generalized $m$-gonal numbers.
More specifically, he investigated for which positive integers $m$ the following congruences hold for all $n \in \Z^+$:
\begin{align}
\label{eq: conj1}
\sum_{k=-\infty}^\infty \sigodd(n-P_m(k)) &\equiv \begin{dcases}
n \pmod 2 & \text{if } n= P_m(j), \, j \in \Z, \\
0 \pmod 2 & \text{otherwise}, \\
\end{dcases}\\
\label{eq: conj2}
\sum_{k=-\infty}^\infty \sigodd(n-P_5(k)) &\equiv \begin{dcases}
n \pmod m & \text{if } n= P_5(j), \, j \in \Z, \\
0 \pmod m & \text{otherwise},
\end{dcases} \\
\label{eq: conj3}
\sum_{k=-\infty}^\infty (-1)^{P_3(-k)}\sigodd(n-P_5(k)) &\equiv \begin{dcases}
(-1)^{P_3(-j)}\cdot n \pmod m & \text{if } n= P_5(j), \, j \in \Z, \\
0 \pmod m & \text{otherwise}.
\end{dcases}
\end{align}

\noindent In particular, Merca posed the following conjectures:
\begin{conj*}
The following are true:
\begin{enumerate}[(i)]
    \item The congruence \eqref{eq: conj1} holds for all $n \in \Z^+$ if and only if $m \in \{5,6\}$. 
    \item The congruence \eqref{eq: conj2} holds for all $n \in \Z^+$ if and only if $m \in \{2,3,6\}$. 
    \item The congruence \eqref{eq: conj3} holds for all $n \in \Z^+$ if and only if $m \in \{2,4\}$. 
\end{enumerate}
\end{conj*}
\noindent Merca showed the if condition for each of these conjectures. Using his work, we obtain the following theorem:

\begin{thm}
Merca's conjectures are true.
\end{thm}

\begin{proof}
We begin by proving (ii). Merca showed that \eqref{eq: conj2} holds if $m \in \{2,3,6\}$ \cite[Theorem~3]{Merca}, hence it suffices to show that if $m \notin\{2,3,6\}$, then there exists some $n \in \Z^+$ such that $n \neq P_5(j)$ for all $j \in \Z$ and  
\begin{equation}\label{eq: conj7}
\sum_{k = -\infty}^\infty \sigodd(n-P_5(k)) \not\equiv 0 \pmod m.
\end{equation}
Since $\sigodd(n-P_5(k)) = 0$ whenever $n-P_5(k) \le 0$, the sum in \eqref{eq: conj7} is in fact finite, and we easily compute $\sum_k \sigodd(3-P_5(k)) = 6$, where $3 \ne P_5(j)$ for $j \in \Z$. Thus, \eqref{eq: conj7} holds unless $6 \equiv 0 \pmod m$. But this is the case only if $m \in \{2,3,6\}$.

Next, we prove (iii). Again, Merca proves that \eqref{eq: conj3} holds if $m \in \{2,4\}$ \cite[Theorem~4]{Merca}, hence it suffices to show that if $m \notin\{2,4\}$, then there exists some $n \in \Z^+$ such that $n\neq P_5(j)$ for all $j \in \Z$ and 
\begin{equation}\label{eq: conj8}
\sum_{k = -\infty}^\infty (-1)^{P_3(-k)} \sigodd(n-P_5(k)) \not\equiv 0 \pmod m.
\end{equation}
We compute $\sum_k (-1)^{P_3(-k)} \sigodd(3-P_5(k)) = 4$, where $3 \ne P_5(j)$ for $j \in \Z$, and so \eqref{eq: conj8} holds unless $4 \equiv 0 \pmod{m}$. But this is the case only if $m \in \{2,4\}$. 

Finally, we prove (i). Since $\sigodd(n)$ is odd if and only if $n$ is a square or a twice square (see \cite[p.~3]{Merca}), we have that 
\begin{equation}\label{eq: mod2}
\sum_{n=1}^\infty \sigodd(n)q^n \equiv \sum_{n=1}^\infty q^{n^2}+\sum_{n=1}^\infty q^{2n^2} \pmod 2.
\end{equation}
The $n$th coefficient of 
\begin{equation*}
\Big(\sum_{\ell=1}^\infty \sigodd(\ell)q^\ell\Big)\Big(\sum_{k=-\infty}^\infty q^{P_m(k)}\Big) = \sum_{\ell=1}^\infty \Big(\sum_{k=-\infty}^\infty \sigodd(\ell)q^{\ell+P_m(k)}\Big) = \sum_{n = 1}^\infty \Big(\sum_{k=-\infty}^\infty \sigodd(n-P_m(k))\Big)q^n
\end{equation*}
is given by 
$\sum_{k=-\infty}^\infty \sigodd(n-P_m(k))$.
On the other hand, the $n$th coefficient of
\begin{align*}
\Big(\sum_{\ell=1}^\infty q^{\ell^2}+\sum_{\ell=1}^\infty q^{2\ell^2}\Big)\Big(\sum_{k=-\infty}^\infty q^{P_m(k)}\Big) = \sum_{\substack{\ell \ge 1 \\ k \in \Z}} \left(q^{\ell^2+P_m(k)}+q^{2\ell^2+P_m(k)}\right)
\end{align*}
is given by 
$a_m(n)+b_m(n)$, where
\begin{align*}
a_m(n) = \abs{A_m(n)} &:= \#\{(\ell,k) \in \Z^+\times \Z:\, \ell^2+P_m(k) = n\}, \\
b_m(n) = \abs{B_m(n)} &:= \#\{(\ell,k) \in \Z^+\times \Z:\, 2\ell^2+P_m(k) = n\}.
\end{align*}
Thus, due to \eqref{eq: mod2} we must have 
\begin{equation*}
\sum_{k=-\infty}^\infty \sigodd(n-P_m(k)) \equiv a_m(n)+b_m(n) \pmod 2.
\end{equation*}

Suppose first that $m \ge 7$. Then we claim that $P_m(0) = 0, P_m(1) = 1$, and $P_m(k) > 3$ for all $k \notin \{0,1\}$. 
From \eqref{eq: m-gonal}, it is clear that $P_m(0) = 0$ and $P_m(1) = 1$. To see that $P_m(k) > 3$ for all $k \notin \{0,1\}$, note that 
since the leading term of $P_m(x)$ is positive and its vertex is at $0 < \frac{m-4}{2m-4} < 1$,
we have $P_m(k) \ge P_m(2) = m \ge 7$ for $k \ge 2$ and $P_m(k) \ge P_m(-1) = m-3 \ge 4$ for $k \le -1$.

Now, let $n=3$. Then the above shows that $n$ is not a generalized $m$-gonal number for $m \ge 7$, and so for \eqref{eq: conj1} to hold, we must have $\sum_{k=-\infty}^\infty \sigodd(3-P_m(k)) \equiv 0 \pmod 2$. If $(\ell,k) \in A_m(3)$, then $\ell^2 = 3-P_m(k)$, so that in particular $\ell^2 \le 3$, which forces $\ell = 1$. But then we must have $P_m(k) = 2$, which we have seen to be impossible. Hence, $A_m(3)$ is empty, and $a_m(3) \equiv 0 \pmod 2$. On the other hand, if $(\ell,k) \in B_m(3)$, we must again have $\ell = 1$. It follows that $P_m(k) = 1$, which is the case if and only if $k = 1$. Hence $B_m(3) = \{(1,1)\}$, and $b_m(3) \equiv 1\pmod 2$. We conclude that $\sum_{k=-\infty}^\infty \sigodd(3-P_m(k)) \equiv a_m(3)+b_m(3)\equiv 1 \not\equiv 0 \pmod 2$.

Merca showed that \eqref{eq: conj1} holds for $m \in \{5,6\}$, and for $m \in \{1,2\}$, the sum in \eqref{eq: conj1} diverges, hence it remains to consider $m \in \{3,4\}$. Suppose first that $m = 3$, and note that $P_3(k) = \frac{1}{2}(k^2+k)$. We have $3 = P_3(-3) = P_3(2)$, so for \eqref{eq: conj1} to hold, we must have $\sum_{k=-\infty}^\infty \sigodd(3-k) \equiv 3 \equiv 1 \pmod 2$. If $(\ell,k) \in A_3(3)$, then $\ell = 1$ and $P_3(k) = 2$, which is impossible. Hence $A_3(3)$ is empty. If $(\ell,k) \in B_3(3)$, then $\ell =1$ and $P_3(k) = 1$, which is the case if and only if $k \in \{-2,1\}$. Thus $B_3(3) = \{(1,-2),(1,1)\}$, and $\sum_{k=-\infty}^\infty \sigodd(3-k) \equiv a_3(3)+b_3(3) \equiv 0 \not\equiv 1 \pmod 2$. 

Finally, suppose $m = 4$, and note that $P_4(k) = k^2$. Since $4 = P_4(2)$, for \eqref{eq: conj1} to hold we must have $\sum_{k=-\infty}^\infty \sigodd(4-k) \equiv 4 \equiv 0 \pmod 2$. If $(\ell,k) \in A_4(4)$, then either $\ell = 1$ and $P_4(k) = 3$, which is impossible, or $\ell = 2$ and $P_4(k) = 0$, which is the case if and only if $k = 0$. Thus, $A_4(4) = \{(2,0)\}$.  On the other hand, if $(\ell,k) \in B_4(4)$, then $\ell =1$ and $P_3(k) = 2$, which is impossible. Thus $B_4(4)$ is empty, and $\sum_{k=-\infty}^\infty \sigodd(4-k) \equiv a_4(4)+b_4(4) \equiv 1 \not\equiv 0 \pmod 2$.
\end{proof}

\bigskip

\section*{Acknowledgements}
We would like to thank Ken Ono for suggesting this project, and for several helpful conversations. We thank William Craig and Badri Pandey, as well as the referee, for their comments on the exposition in this note. Finally, we are grateful for the generous support of the National Science Foundation (DMS 2002265 and DMS 205118), the National Security Agency (H98230-21-1-0059), the Thomas Jefferson Fund at the University of Virginia, and the Templeton World Charity Foundation.

\bigskip


\bigskip

\begin{thebibliography}{99}

\bibitem{BM}{Cristina M. Ballantine and Mercia Merca, \emph{Parity of Sums of Partition Numbers and Squares in Arithmetic Progressions}, Ramanujan J. \textbf{44} (2017), 617-630.}

\bibitem{HZ}{Letong Hong and Shengtong Zhang, \emph{Proof of the {B}allentine-{M}erca Conjecture and Theta Function Identities Modulo 2}, \url{https://arxiv.org/abs/2101.09846}, 2021.}

\bibitem{Merca}{Mircea Merca, \emph{Congruence Identities Involving Sums of Odd Divisors Function}, Proc. Rom. Acad. Ser. A \textbf{22}, no. 2 (2021), 119-125.}

\end{thebibliography}
\end{document}